\documentclass[11pt,a4paper]{article}

\usepackage{amsmath}    
\usepackage{amssymb}    
\usepackage{amsthm}     
\usepackage{graphics}

\usepackage{hyperref}
\hypersetup{
breaklinks,
colorlinks=true,
linkcolor=blue,
pdfborder = 0 0 0,
citecolor=blue,
urlcolor=blue}

\numberwithin{equation}{section}

\makeatletter\@twosidetrue\makeatother

\newtheorem{theorem}{Theorem}[section]
\newtheorem{lemma}[theorem]{Lemma}

\newtheorem{proposition}[theorem]{Proposition}

\theoremstyle{definition}
\newtheorem{definition}[theorem]{Definition}
\newtheorem{remark}[theorem]{Remark}

\def\DH{D \! H}

\title{\vspace{-10mm}
\Large\textbf{An overdetermined problem associated\\ to the Finsler Laplacian}}
\author{\medskip\normalsize Giulio Ciraolo and Antonio Greco}

\date{}

\begin{document}

\belowdisplayshortskip = \belowdisplayskip

\paperwidth=210 true mm
\paperheight=297 true mm
\pdfpagewidth=210 true mm
\pdfpageheight=297 true mm

\maketitle

\vspace{-5mm}

\pdfbookmark[2]{Abstract}{Abstract}
\begin{abstract}
We prove a rigidity result for the anisotropic Laplacian. More precisely, the domain of the problem is bounded by an unknown surface supporting a Dirichlet condition together with a Neumann-type condition which is not translation-invariant. Using a comparison argument, we show that the domain is in fact a Wulff shape. We also consider the more general case when the unknown surface is required to have its boundary on a given conical surface: in such a case, the domain of the problem is bounded by the unknown surface and by a portion of the given conical surface, which supports a homogeneous Neumann condition. We prove that the unknown surface lies on the boundary of a Wulff shape.
\end{abstract}

\section{Introduction}

In this manuscript we study an overdetermined boundary value problem for elliptic equations. In these kinds of problems, a well posed elliptic PDEs problem is overdetermined by adding a further condition on the solution at the boundary and, for this reason, a solution may exists only if the domain and the solution itself satisfy some suitable symmetry. For instance, the well-known Serrin's overdetermined problem deals with the torsion problem in a bounded domain $\Omega \subseteq\mathbb R^N$
\begin{equation} \label{pb_torsion}
\begin{cases}
\Delta u = -1 & \text{ in } \Omega \,, \\
u=0 & \text{ on } \partial \Omega\,,
\end{cases}
\end{equation}
with the overdetermining condition 
\begin{equation}\label{nabla_u_const}
|Du | = c \quad  \text{ on } \partial \Omega \,,
\end{equation}
for some positive constant $c$. Hence, Problem \eqref{pb_torsion}-\eqref{nabla_u_const} is not well-posed and a solution may exists only if the domain (and the solution itself) satisfies some special symmetry (radial symmetry in this case, see \cite{Serrin}).  There are many other results concerning overdetermined problems and, in particular, many generalizations of problem \eqref{pb_torsion}-\eqref{nabla_u_const} have been considered in recent years, such as for quasilinear operators, for domains in convex cones, and in a Finsler (or anisotropic) setting (see for instance \cite{Bianchini-Ciraolo,Cianchi-Salani,FarinaKawhol,FGK,Garofalo-Lewis,Wang-Xia} and references therein). 

The anisotropic setting that we are considering can be described in terms of a norm $H_0$ in~$\mathbb R^N$.
Let  $H$ be the dual norm of~$H_0$
(see Section~\ref{norms}), and consider the Finsler Laplacian~$\Delta_H$, whose definition is recalled in Section~\ref{Finsler}.
Under convenient assumptions, Cianchi and Salani \cite[Theorem~2.2]{Cianchi-Salani} generalized Serrin's result to this setting
and proved that the
translation-invariant overdetermined problem
\begin{equation}\label{CS_ODP}
\begin{cases}
-\Delta_{H \,} u = 1 &\mbox{in $\Omega$}
\\
u = 0 &\mbox{on $\partial \Omega$}
\\
H(Du(x)) = \mbox{const.} &\mbox{on $\partial \Omega$}
\end{cases}
\end{equation}
is solvable if and only if $\Omega = B_R(x_0,H_0)$ for some $x_0 \in
\mathbb R^N$ and $R > 0$ (see also \cite[Theorem~1.1]{Bianchini-Ciraolo} for the generalization to anisotropic $p$-Laplace equations).

If the overdetermining condition is not prescribed on the whole boundary, then the problem is called partially overdetermined. In this case, one can say less on the solution and a large variety of situations may occur. For instance, if we relax problem \eqref{pb_torsion}-\eqref{nabla_u_const} by prescribing the Dirichlet condition $u = 0$ on a proper subset $\Gamma_0 \subseteq\partial \Omega$ instead that on the whole boundary, then the existence of a solution does not imply that $\Omega$ is a ball: the simplest counterexample is given by the annulus, and more refined counterexamples are found in~\cite{FGLP}. Nevertheless, under convenient additional assumptions, a partially overdetermined problem turns out to be globally overdetermined and the conclusion can be recovered (see \cite{FV1,FV2,FG}).

In this paper we consider an anisotropic overdetermined problem in cones. Let\/ $\Omega \subseteq \mathbb R^N$, $N \ge 2$, be a bounded domain (i.e. an open, connected, nonempty subset) containing the origin~$O$, and let $\Sigma \subseteq \mathbb R^N$ be a cone
\begin{equation}\label{cone}
\Sigma=\lbrace\, tx \, : \, x\in\omega, \, t\in(0,+\infty) \,\rbrace
\end{equation}
for some domain $\omega \subseteq S^{N-1}$. We mention that the equality $\omega = \mathbb S^{N-1}$ (which implies $\Sigma = \mathbb R^N$) is allowed throughout the paper. In the case when $\omega \subsetneq \mathbb S^{N-1}$ we require that $\partial \Sigma \setminus \{O\}$ is a hypersurface of class~$C^1$ and therefore possesses an outward normal~$\nu$. Define
\begin{equation}\label{Gamma01_def}
\Gamma_0 = \Sigma \cap \partial \Omega \quad \text{ and } \quad \Gamma_1 = \partial\Omega \setminus \overline \Gamma_0 \,.
\end{equation}
Several problems in convex cones have been considered recently, like the isoperimetric and Sobolev inequalities in convex cones (see \cite{Cabre-RosOton-Serra, CirFigRon, Lions-Pacella, Lions-Pacella-Tricarico}) and overdetermined and Liouville type problems in \cite{CirFigRon,CiraoloRoncoroni, Pacella-Tralli}. 

Here we extend the approach in \cite{Greco_AMSA} to the more general anisotropic setting and by considering a (possibly) mixed boundary-value problem. The starting point lies in the observation (done in \cite[p.~28]{CiraoloRoncoroni}) that the solution of~\eqref{pb_torsion} in the Euclidean ball $\Omega = B_R(O,\,\allowbreak|\cdot|)$ obviously satisfies (being a radial function) $u_\nu = 0$ along $\Gamma_1 \setminus \{O\}$ for every smooth cone $\Sigma \subsetneq \mathbb R^N$. Our main result is the following.

\begin{theorem} \label{thm_main1}
Let\/ $\Omega$ and\/ $\Sigma \subseteq \mathbb R^N$ be as above, and let\/ $H$ be a norm of class $C^1(\mathbb R^N \setminus \{O\})$
such that the function $V(\xi) = \frac12 \, H^2(\xi)$ is strictly
convex. Let $q(r)$ be a positive, real-valued function such that the ratio
$q(r)/r$ is increasing in $r > 0$. If there exists a weak solution $u \in C^1\big((\Sigma \cap \Omega) \cup (\Gamma_1 \setminus \{O\}\big) \cap C^0\big(\,\overline{\Sigma \cap \Omega} \setminus \{O\}\big)$ of the problem
\begin{equation}\label{special_pb}
\begin{cases}
-\Delta_{H \,} u = 1 &\mbox{in $\Sigma \cap \Omega$}
\\
u = 0 &\mbox{on $\Gamma_0$}
\\
DV(Du(x)) \cdot \nu = 0 &\mbox{on $\Gamma_1 \setminus \{O\}$} \,,
\end{cases}
\end{equation}
satisfying the condition
\begin{equation} \label{overd_cond}
\lim_{x \to z} H(Du(x)) = q(H_0(z)) \quad \text{ for all $z \in \overline \Gamma_0$,}
\end{equation}
then\/ $\Sigma \cap \Omega=\Sigma \cap B_R(O,H_0)$ for some $R > 0$. 
\end{theorem}

In the case when $\Sigma = \mathbb R^N$ we have $\Gamma_1 = \emptyset$ and the third condition in~\eqref{special_pb} is trivially satisfied. If, in addition to $\Sigma = \mathbb R^N$, we also have $H(\cdot) = |\cdot|$, then Theorem \ref{thm_main1} was proved by the second author in \cite{Greco_AMSA} under the weaker assumption that $q(r)/r$ is non-decreasing (see also \cite{Greco_PM,Greco-Mascia}). We mention that the rate of growth of $q$ is crucial to obtain the rigidity result. A counterexample for the Euclidean norm can be found in \cite[p.~488]{Greco_PM}. We also notice that, in the Euclidean case, the boundary condition on $\Gamma_1$ is simply $u_\nu = 0$.

We stress that problem \eqref{special_pb}-\eqref{overd_cond} can be seen as a partially overdetermined problem, since the overdetermining condition is given only on the part~$\overline \Gamma_0$ of the boundary. Accordingly, we are able to determine the shape of~$\Gamma_0$, while $\Gamma_1$ depends on the choice of the cone~$\Sigma$.

We emphasize that no regularity assumption is imposed on $\Gamma_0$. For this reason, we have to consider condition \eqref{overd_cond} instead of the simpler 
$$
H(Du(z)) = q(H_0(z)) \quad \textmd{ on } \Gamma_0
$$ 
(as, for instance, in \cite{Garofalo-Lewis}).

We also mention that Theorem \ref{thm_main1} could be extended to the case in which the ratio
$q(r)/r$ is non-decreasing in $r > 0$ by using Hopf's boundary point lemma (see \cite{Greco_PM}), as well as to more general anisotropic quasilinear operators (see for instance \cite{Roncoroni}). More precisely, one has to prove a 
Hopf's boundary point comparison principle between the solution and the solution in the Wulff shape. In this direction, the results in \cite{CastorinaRieySciunzi} can be a starting point for this investigation, and one can expect to prove a symmetry result in cones in the spirit of Theorem \ref{thm_main1} for a class of anisotropic equations of the form ${\rm div} (DV(Du)) + f(u)=0$. 

\medskip

The paper is organized as follows.  In Section \ref{norms} we recall some well-known facts about norms in $\mathbb R^N$. In Section \ref{Finsler} we recall the definition of Finsler Laplacian and prove some basic properties of \eqref{special_pb}. Sections \ref{norms} and \ref{Finsler} will be the occasion to give full details of some basic facts, and for this reason we give a detailed description which is readable also at a beginner level. In Section \ref{proof} we give the proof of Theorem \ref{thm_main1}. In Appendix \ref{example} we provide an example of a \emph{smooth} norm having \emph{non-smooth} dual norm (see also \cite[Example A. 1.19]{Cannarsa-Sinestrari}).

\vspace{2em}

\section{Norms, dual norms and Wulff shapes}\label{norms}

In this section we collect the definitions and properties needed in
the sequel. Further details are found in many recent papers: see, for
instance, \cite[Section~2.2]{Bianchini-Ciraolo},
\cite[Section~2.1]{Bellettini-Paolini}, and
\cite[Section~2.3]{Crasta-Malusa}. Standard references on convex
analysis are \cite{Rockafellar} and~\cite{Schneider} (see also
\cite[Section~5.3]{Simon}).

\subsection{Norms, convexity, and the Wulff shape}\label{convexity}

Let $H_0 \colon \mathbb R^N \to \mathbb R$ be a norm on~$\mathbb R^N$,
$N \ge 1$, i.e.\ let $H_0$ be a nonnegative function such that
\begin{align}
\label{zero}
&\mbox{$H_0(x) = 0$ if and only if $x = 0$;}
\\
\noalign{\medskip}
\label{homogeneity}
&\mbox{$H_0(tx) = |t| \, H_0(x)$ for all $t \in
  \mathbb R$ and $x \in \mathbb R^N$;}
\\
\noalign{\medskip}
\label{triangle}
&\mbox{$H_0(x_1 + x_2) \le H_0(x_1) + H_0(x_2)$ for all
  $x_1, x_2 \in \mathbb R^N$.}
\end{align}
The last inequality, known as \textit{the triangle inequality}, may be
equivalently replaced by the requirement that $H_0$ is \textit{a
  convex function}, as in~\cite{Bianchini-Ciraolo,BCS,Cianchi-Salani}.
Indeed, from (\ref{homogeneity})-(\ref{triangle}) it follows that every
norm satisfies $H_0(\lambda x_1 + (1 - \lambda) \, x_2) \le \lambda \,
H_0(x_1) + (1 - \lambda) \, H_0(x_2)$ for all $\lambda \in (0,1)$ and
$x_1,x_2 \in \mathbb R^N$, i.e.\ $H_0$~is a convex function.
Conversely, every nonnegative, convex function $H_0 \colon \mathbb
R^N \to \mathbb R$ satisfying (\ref{zero}) and~(\ref{homogeneity}) also
satisfies~(\ref{triangle}): indeed, we may write $H_0(x_1 + x_2) =
H_0((2x_1 + 2x_2)/2)$, and by convexity $H_0(x_1 + x_2) \le \frac12 \,
H_0(2x_1) + \frac12 \, H_0(2x_2)$. Now using~(\ref{homogeneity}) we
arrive at~(\ref{triangle}) and hence $H_0$ is a norm
(cf.~\cite[Theorem~5.3.8]{Simon}).

We denote by $B_R(x_0,H_0) = \{\, x \in \mathbb R^N : H_0(x - x_0) < R
\,\}$ the ball centered at~$x_0$ and with radius~$R > 0$ with respect to
the norm~$H_0$ (also called the \textit{Wulff shape}).

\subsection{Dual norms}

As usual, the dual norm $H(\xi)$ of the norm $H_0(x)$ is defined for
$\xi \in \mathbb R^N$ by
\begin{equation}\label{dual}
H(\xi) = \sup_{x \ne 0}
\frac{\, x \cdot \xi \,}{\, H_0(x) \,}
.
\end{equation}
It is well known that the supremum above is indeed a maximum, i.e., it
is attained with a particular $x \ne 0$. Furthermore, any given norm
$H_0$ turns out to be the dual norm of its dual norm~$H$: see, for
instance, \cite[Corollary~1.4]{Brezis}.

\subsection{Properties of the gradient of a norm}

Let us recall some essential properties of the gradient $\DH_0$ of a
(differentiable) norm~$H_0$.

\begin{lemma}\label{essential}
If $H_0$ is differentiable at some $x \in \mathbb R^N \setminus \{O\}$, then
\begin{enumerate}
\item The scalar product $x \cdot \DH_0(x)$ satisfies
\begin{equation}\label{radial_derivative}
x \cdot \DH_0(x) = H_0(x)
.
\end{equation}
\item $H_0$ is differentiable at $tx$ for every $t \in \mathbb R
  \setminus \{\, 0 \,\}$, and satisfies
\begin{equation}\label{differential}
\DH_0(tx) = (\mathop{\rm sgn} t) \, \DH_0(x)
.
\end{equation}
\item The gradient $\DH_0(x)$ is a unit vector with respect to the dual norm~$H$ in the sense that
\begin{equation}\label{unit}
H(\DH_0(x)) = 1
.
\end{equation}
\end{enumerate}
\end{lemma}

\begin{proof}
Property~(\ref{radial_derivative}) is already found in the seminal
dissertation by Finsler (cf.~\cite[(28)]{Finsler}) as well as in
several recent papers: see, for instance,
\cite[(2.10)]{Bellettini-Paolini}, \cite[(1.8)]{Ferone-Kawohl} and
\cite[Proposition~1, (i)]{Wang-Xia}. Equality~(\ref{differential}),
instead, can be derived from~\cite[(2.9)]{Bellettini-Paolini}. Let us
give a proof, for completeness. Both (\ref{radial_derivative})
and~(\ref{differential}) are obtained by differentiating the
equality~(\ref{homogeneity}): more precisely,
(\ref{radial_derivative}) follows by
differentiating~(\ref{homogeneity}) in~$t$ at $t = 1$,
while~(\ref{differential}) is obtained by differentiation in~$x_i$ for
$i = 1, \ldots, N$.
Equality~\ref{unit} corresponds to~\cite[(3.12)]{Cianchi-Salani},
\cite[(1.7)]{Ferone-Kawohl} and \cite[Proposition~1,
(iii)]{Wang-Xia}). Let us prove the assertion and give a geometrical
interpretation.
\begin{figure}[h]
\centering
\begin{picture}(200,130)(55,8)
\includegraphics{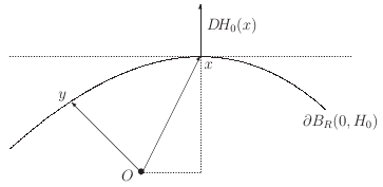}
\end{picture}
\caption{Maximizing the scalar product $y \cdot \DH_0(x)$ under $H_0(y) = R$}
\label{dual_norm}
\end{figure}

\noindent Define $R = H_0(x)$ and consider the ball
$B_R(O,H_0)$. By~(\ref{dual}), in order to compute the dual norm
$H(\DH_0(x))$ it suffices to find a point $y \in \partial B_R(O,H_0)$
that maximizes the ratio
$$
\frac{\, y \cdot \DH_0(x) \,}{\, H_0(y) \,}
=
\frac{\, y \cdot \DH_0(x) \,}{\, R \,}
.
$$%
The hyperplane passing through~$x$ and orthogonal to~$\DH_0(x)$ is a
supporting hyperplane for the convex set $B_R(O,H_0)$, and hence we
may take $y = x$ to maximize the numerator (see
Figure~\ref{dual_norm}). Consequently we have
$$
H(\DH_0(x))
=
\frac{\, x \cdot \DH_0(x) \,}{\, R \,}
.
$$%
Now using~(\ref{radial_derivative}) the conclusion follows.
\end{proof}

\subsection{Differentiability of a norm}

Let $H_0$ be a norm, and denote by $H$ its dual norm. Because
of~(\ref{homogeneity}), $H_0(x)$ is never differentiable at $x = 0$.
By~\cite[Corollary~1.7.3]{Schneider}, instead, differentiability
of~$H_0$ at $x \ne 0$ is related to the strict convexity of the unit
ball $B_1(0,H)$ of the dual norm~$H$: the next lemma collects several
equivalent conditions.

\begin{lemma}\label{differentiability}
The  following conditions are equivalent.
\begin{enumerate}
\def\labelenumi{\rm(\arabic{enumi})}
\def\theenumi{(\arabic{enumi})}
\item\label{1} $H_0(x)$ is differentiable at every $x \ne 0$.
\item\label{2} $H_0 \in C^1(\mathbb R^N \setminus \{O\})$.
\item\label{3} $B_1(0,H_0)$ is a domain of class $C^1$.
\item\label{4} $B_1(0,H)$ is strictly convex.
\item\label{5} The function $V(\xi) = \frac12 \, H^2(\xi)$ is strictly
  convex.
\end{enumerate}
\end{lemma}

\begin{proof}
$\ref{2} \Rightarrow \ref{1}$ is obvious. The converse implication
$\ref{2} \Leftarrow \ref{1}$ follows by \cite[Corollary
25.5.1]{Rockafellar} because $H_0$ is a convex function. The
implication $\ref{2} \Rightarrow \ref{3}$ holds because $\DH_0(x) \ne
0$ when $x \ne 0$ by homogeneity~(\ref{homogeneity}), and hence the
level surface $\Gamma_1 = \partial B_1(0,H_0) = \{\, x \in \mathbb R^N
: H_0(x) = 1 \,\}$ is of class~$C^1$ by the implicit function
theorem. To prove the converse implication $\ref{2} \Leftarrow
\ref{3}$, observe that every $x \ne 0$ can be represented in polar
coordinates $\rho,\eta$ given by $\rho = | x |$ and $\eta =
| x |^{-1} \, x$. In such a coordinate system, by~\ref{3} and
by the convexity of~$B_1(0,H_0)$ the surface $\Gamma_1 = \partial
B_1(0,H_0)$ is the graph of a $C^1$-function $\rho = \rho(\eta)$ whose
domain is the Euclidean unit sphere $\mathbb S^{N - 1} = \{\, \eta \in
\mathbb R^N : | \eta | = 1 \,\}$. Since $x/H_0(x) \in
\Gamma_1$, we may write
$$
\frac{x}{\, H_0(x) \,}
=
\rho(\eta) \, \eta
=
\rho(| x |^{-1} \, x)
\,
\frac{x}{\, | x | \,}
$$%
and hence $H_0(x) = | x | / \rho(| x |^{-1} \, x)$.
Consequently, $H_0 \in C^1(\mathbb R^N \setminus \{O\})$ and
\ref{2} holds, hence $\ref{2} \Leftarrow \ref{3}$. Let us check that
$\ref{1}$ is equivalent to~$\ref{4}$. As a consequence
of~\cite[Corollary~1.7.3]{Schneider}, the unit ball $B_1(0,H)$ of the
dual norm~$H$ is strictly convex (its boundary does not contain any
segment) if and only if $H_0$ is differentiable at every $x \in
\mathbb R^N \setminus \{O\}$, namely $\ref{1} \Leftrightarrow
\ref{4}$, as claimed. The preceding arguments imply that the first
four conditions in the statement are equivalent to each other. To
complete the proof we now verify that $\ref{4}$ is equivalent
to~$\ref{5}$. Before proceeding further, observe that $2V(\xi)$ (hence
$V(\xi)$ as well) is convex because it is the square of the convex,
nonnegative function $H(\xi)$. Hence $V(\xi)$ is not strictly convex
if and only if there exists a line segment~$\ell \subseteq\mathbb R^N$
such that the restriction $V|_\ell$ is a linear function. But then
$H(\xi) = \sqrt{2V(x) \,}$ is concave along~$\ell$. Since $H(\xi)$ is
also convex, it follows that $H(\xi)$ is constant along~$\ell$. Hence
$V(\xi)$ fails to be strictly convex if and only if $H(\xi)$ takes a
constant value (say $c$) along some segment~$\ell$. Now recall that
$H(\xi)$ is homogeneous of degree~$1$: this has two relevant
consequences. The first consequence is that $H(\xi)$ is not constant
in the radial direction, hence the segment~$\ell$ is not aligned with
the origin. The second consequence is that we may find a new segment,
say $\ell'$, parallel to~$\ell$ and such that $H(\xi) = 1$ on~$\ell'$:
indeed, we may take
$$
\ell'
=
\{\,
\xi \in \mathbb R^N
:
c \, \xi \in \ell
\,\}
.
$$
In short, there is no loss of generality if we assume $c = 1$. But
then we may assert that $V(\xi)$ is not strictly convex if and only if
$\partial B_1(0,H)$ contains a line segment, i.e., if and only if
$B_1(0,H)$ is not strictly convex. This proves the equivalence between
$\ref{4}$ and~$\ref{5}$, and the lemma follows.
\end{proof}

It may well happen that $H_0 \in C^1(\mathbb R^N \setminus \{O\})$ and $H \not \in C^1(\mathbb R^N \setminus \{O\})$: see Section~\ref{example} or \cite[Example A. 1.19]{Cannarsa-Sinestrari}. If both  $H_0$ and~$H$ are smooth, then it is relevant for our purposes to notice that the gradient $\DH(\xi)$ evaluated at $\xi = \DH_0(x)$ is radial in the following sense (see also \cite[Lemma~2.2]{Bellettini-Paolini} and~\cite[c), p.~249]{Ferone-Kawohl}):

\begin{lemma}\label{radial}
If $H_0,H \in C^1(\mathbb R^N \setminus \{O\})$ then $x = H_0(x) \, \DH(\DH_0(x))$ for all $x \in \mathbb R^N \setminus \{O\}$. Furthermore, for every $\xi \ne O$, a point $x \ne O$ realizes the supremum in~\eqref{dual} if and only if
\begin{equation}\label{realizes}
x = H_0(x) \, \DH(\xi)
.
\end{equation}
\end{lemma}

\begin{proof}
The unit ball $B_1(0,H_0)$ is a strictly convex domain of class~$C^1$, hence for every $\xi \in \partial B_1(0,H)$ there exists a unique $x \in \partial B_1(0,H_0)$ that maximizes the linear function $L(x) = x \cdot \xi$ under the constraint $H_0(x) = 1$. Taking the definition~\eqref{dual} of~$H(\xi)$ into account, we may say that for every $\xi \in \partial B_1(0,H)$ there exists a unique~$x \in \partial B_1(0,H_0)$ satisfying the equality
$$
1 = x \cdot \xi
.
$$
Furthermore there exists $\lambda \in \mathbb R$ (the Lagrange multiplier) such that $\xi = \lambda \, \DH_0(x)$. More precisely, since $x \cdot \DH_0(x) > 0$ by the convexity of~$B_1(0,H_0)$, we have~$\lambda > 0$. This and~\eqref{unit} imply $H(\xi) = \lambda$, and therefore $\xi = H(\xi) \, \DH_0(x)$. By reverting the roles of $x$ and~$\xi$ we get~\eqref{realizes} and the lemma follows using~\eqref{differential}.
\end{proof}

We conclude this section with the following proposition:

\begin{proposition}[Regularity of the Lagrangian]\label{smoothness}
Let $H$ be a norm of class $C^1(\mathbb R^N \setminus \{O\})$.
Then the Lagrangian $V(\xi) = \frac12 \, H^2(\xi)$ belongs to the class
$C^1(\mathbb R^N)$
\end{proposition}

\begin{proof}
Since all norms on~$\mathbb R^N$ are equivalent, there exist two
positive constants $\sigma,\gamma$ such that
\begin{equation}\label{equivalent}
\sigma \, | \xi |
\le
H(\xi)
\le
\gamma \, | \xi |
\mbox{ for all $\xi \in \mathbb R^N$}
\end{equation}
(the notation is taken from~\cite[(3.4)]{Cianchi-Salani}). In
particular, $H(\xi)$ is continuous at $\xi = 0$. Concerning the
differentiability, for $\xi \ne 0$ we may apply the standard rules of
calculus and get
\begin{equation}\label{DV}
DV(\xi) = H(\xi) \allowbreak \, \DH(\xi)
.
\end{equation}
Furthermore, the right-hand side admits a continuous extension to $\xi
= 0$ because $\DH(\xi)$ is bounded on the compact surface $\mathbb
S^{N - 1} = \{\, \xi : | \xi | = 1 \,\}$, and
by~(\ref{differential}) $\DH(\xi)$ is also bounded in the whole
punctured space $\mathbb R^N \setminus \{O\}$. Hence
$DV(\xi) \to 0$ as $\xi \to 0$. This and the continuity of $V(\xi)$ at $\xi
= 0$ imply that $V$ is also differentiable at $\xi = 0$, and $DV(0) =
0$. The proposition follows.
\end{proof}

\section{The Finsler Laplacian}\label{Finsler}

Given a norm $H$ of class $C^1(\mathbb R^N \setminus \{O\})$,
the function $V(\xi) = \frac12 \, H^2(\xi)$ belongs to the class
$C^1(\mathbb R^N)$ by Proposition~\ref{smoothness}. The
\textit{Finsler Laplacian} associated to~$H$ is the differential
operator $\Delta_H$ which is formally defined by
$$
\Delta_{H \,} u(x) = {\rm div}
\Big( DV(Du(x)) \Big)
.
$$%

 \subparagraph{Notation.}

In the present paper it is understood that the gradient operator~$D$
takes precedence over the composition of functions: thus, the notation
$DV(Du(x))$ represents the vector field $DV(\xi)$ evaluated at the
point $\xi = Du(x)$. Such a vector field differs, in general, from the
field whose components are the derivatives of $V(Du(x))$ with respect
to~$x_i$, $i = 1, \dots, N$. Clearly, if $H(\xi)$ is
the Euclidean norm $| \xi |$ then $V(\xi) = \frac12 \, |
\xi |^2$ and therefore $DV(\xi) = \xi$. Thus, the operator
$\Delta_H$ reduces to the standard Laplacian~$\Delta$.

Let $\Omega$ be a bounded domain in $\mathbb R^N$, $N \ge 2$, containing the origin $O$. Let $\Sigma$ and $\Gamma_0,\Gamma_1$ be as in~\eqref{cone} and~\eqref{Gamma01_def}, respectively. We define the function space
$$
W_{\Gamma_0}^{1,2} (\Omega \cap \Sigma) = \{ v\colon \Omega \cap \Sigma \to \mathbb R \ \textmd{ s.t. } \  v = w \, \chi_{\Omega \cap \Sigma} \textmd{ for some } w \in W^{1,2}_0(\Omega) \} \,,
$$
where $\chi_{\Omega \cap \Sigma}$ denotes the characteristic function of $\Omega \cap \Sigma$. Notice that a function $v$ in $W_{\Gamma_0}^{1,2} (\Omega \cap \Sigma)$ has zero trace on $\Gamma_0$.

\begin{definition}[Weak solution]\label{definition}
Let $\Omega$ be as above and let $f$ be a function
in $L^2(\Omega \cap \Sigma)$. A \textit{weak solution} of
\begin{equation}\label{Dirichlet}
\begin{cases}
-\Delta_{H \,} u = f &\mbox{in $\Omega \cap \Sigma$;}
\\
u = 0 &\mbox{on $\Gamma_0$}
\\
\noalign{\vspace{1pt}}
DV(Du) \cdot \nu = 0 &\mbox{on $\Gamma_1 \setminus \{O\}$}
\end{cases}
\end{equation}
is a function $u \in W_{\Gamma_0}^{1,2}(\Omega \cap \Sigma)$ such that 
\begin{equation}\label{weak}
\int_{\Omega \cap \Sigma}
Dv(x) \cdot DV(Du(x)) \, dx
=
\int_{\Omega \cap \Sigma}
f(x) \, v(x) \, dx
\end{equation}
for every $v \in W_{\Gamma_0}^{1,2}(\Omega\cap \Sigma)$.
\end{definition}

\begin{theorem}[Existence]\label{existence}
Let\/ $\Omega$, $\Sigma$ and $f$ be as above. If\/ $H$ is a norm of class $C^1(\mathbb R^N \setminus \{O\})$, then Problem\/~{\rm(\ref{Dirichlet})} has a weak solution.
\end{theorem}

\begin{proof}
Define $V(\xi) = \frac12 \, H^2(\xi)$. By~(\ref{equivalent}), and by the Poincar\'e inequality in~$W^{1,2}_{\Gamma_0}(\Omega \cap \Sigma)$ (see \cite[Theorem~7.91]{Salsa}), the functional
\begin{equation}\label{F}
F[u]
=
\int_{\Omega \cap \Sigma}
\Big(
V(Du(x))
-
f(x) \, u(x)
\Big)
\,
dx
\end{equation}
is well defined and coercive over the Sobolev space $W^{1,2}_{\Gamma_0}(\Omega \cap \Sigma)$, hence there exists a minimizer. Since the
functional~$F$ is differentiable (as a consequence of
Proposition~\ref{smoothness}), each minimizer is a weak solution of
the Euler equation $-\Delta_{H \,} u = f(x)$.
\end{proof}

\begin{remark}
 If, in addition to the assumption of Theorem~\ref{existence},
  the function $V(\xi)$ is strictly convex, then the
  functional~$F$ in~(\ref{F}) is also strictly convex, and the
  minimizer is unique. Uniqueness of the weak solution to
  Problem~(\ref{Dirichlet}) also follows by letting $\Omega^1 = \Omega^2 = \Omega$ in Lemma~\ref{monotonicity}. Several
  conditions equivalent to the strict convexity of $V(\xi)$ are given in Lemma~\ref{differentiability}.
\end{remark}

In view of our subsequent application we now prepare the following
comparison principle, which asserts that if $f \ge 0$ then the
solution of~(\ref{Dirichlet}) is not only unique but also nonnegative and
monotonically increasing with respect to set inclusion.

\begin{lemma}[Nonnegativity] \label{nonnegativity}
Let\/ $H$ be a norm of class $C^1(\mathbb R^N \setminus \{O\})$, and let\/ $\Omega$ and\/~$\Sigma$ be as above. If $f \in L^2(\Omega \cap \Sigma)$ is nonnegative, then any weak solution $u$ to \eqref{Dirichlet} is also nonnegative.
\end{lemma}

\begin{proof}
By~\eqref{DV} and~\eqref{radial_derivative} we find $\xi \cdot DV(\xi) = 2V(\xi)$ for $\xi \ne 0$. The equality continues to hold at $\xi = 0$ by Proposition~\ref{smoothness}. Hence, using $v(x) = -u^-(x) = \min\{u(x),0\}$ as a test-function in \eqref{weak} we get
$$
0
\le
2\int_{\Omega \cap \Sigma} V(Dv(x)) \, dx
=
2\int_{\{\,u<0\,\}} V(Du(x)) \, dx
=
\int_{\{\,u<0\,\}} f(x) \, u(x) \, dx
\le
0\,,
$$ 
which implies $Dv(x) = 0$ almost everywhere in $\Omega \cap \Sigma$. Since $v \in W^{1,2}_{\Gamma_0}(\Omega \cap \Sigma)$, by the Poincar\'e inequality \cite[Theorem 7.91]{Salsa} it follows that $v = 0$, hence $u \ge 0$ a.e.\ in~$\Omega \cap \Sigma$.
\end{proof}

\begin{lemma}[Monotonicity]\label{monotonicity} Let\/ $H$ be a norm of class $C^1(\mathbb R^N \setminus \{O\})$ such that the function $V(\xi) = \frac12 \, H^2(\xi)$ is strictly convex. Let\/~$\Sigma$ be as in~\eqref{cone}, and let\/ $\Omega^i$, $i = 1,2$, be two bounded domains in\/~$\mathbb R^N$, $N \ge 2$, containing the origin and satisfying\/ $\Omega^1 \cap \Sigma \subseteq \Omega^2 \cap \Sigma$. Choose a nonnegative $f \in L^2(\Omega^2 \cap \Sigma)$, and denote by $u_i$ any weak solution of Problem\/~{\rm(\ref{Dirichlet})} with\/ $\Omega = \Omega^i$. Then $u_1 \le u_2$ almost everywhere in\/~$\Omega^1$.
\end{lemma}

\begin{proof}
Let $\Gamma_{\! 0}^i = \Sigma \cap \partial \Omega^i$, $i = 1,2$. Since $f \ge 0$, from Lemma \ref{nonnegativity} we have $u_2 \ge 0$ a.e.\ in~$\Omega_2$. Hence the function $v = (u_1 - u_2)^+$ belongs to $W^{1,2}_{\Gamma_{\! 0}^1}(\Omega^1 \cap \Sigma)$ and has an extension, still denoted by~$v$, to $W^{1,2}_{\Gamma_{\! 0}^2}(\Omega^2 \cap \Sigma)$ vanishing identically outside $\Omega^1 \cap \Sigma$. Therefore $v$ is an admissible test-function in Definition~\ref{definition} for $\Omega = \Omega^i$, $i = 1,2$, and we may write
\begin{align*}
\int_{\Omega^1 \cap \Sigma}
Dv(x) \cdot DV(Du_1(x)) \, dx
&=
\int_{\Omega^1 \cap \Sigma}
f(x) \, v(x) \, dx,
\\
\noalign{\medskip}
\int_{\Omega^1 \cap \Sigma}
Dv(x) \cdot DV(Du_2(x)) \, dx
&=
\int_{\Omega^1 \cap \Sigma}
f(x) \, v(x) \, dx.
\end{align*}
By subtracting the second equality from the first one we obtain
$$
\int_{\{\, v > 0 \,\}}
\Big(Du_1(x) - Du_2(x)\Big)
\cdot
\Big(DV(Du_1(x)) - DV(Du_2(x))\Big)
\,
dx
=
0
$$%
Since $V$ is strictly convex by assumption, the Lebesgue measure of the set $\{\, v > 0 \,\}$ must be zero, and the lemma follows.
\end{proof}

In the case when $\Omega = B_R(O,H_0)$ for some $R > 0$ and $f \equiv
1$, Problem~(\ref{Dirichlet}) is explicitly solvable:

\begin{proposition}[Solution in the Wulff shape]\label{solution_WS}
Let\/ $H$ be a norm of class $C^1(\mathbb R^N \setminus \{O\})$,
and suppose that its dual norm $H_0$ also belongs to $C^1(\mathbb R^N \setminus \{O\})$. Let $\Sigma$ be as in~\eqref{cone}. The function $u_R \in C^1(\mathbb R^N)$ given by $u_R(x) = \frac1{\, 2N \,} \, (R^2 - H_0^2(x))$ is a weak solution of the problem
\begin{equation}\label{radial_Dirichlet}
\begin{cases}
-\Delta_{H \,} u = 1 &\mbox{in $B_R(O,H_0) \cap \Sigma$,}
\\
\noalign{\medskip}
u = 0 &\mbox{on $\Sigma \cap \partial B_R(O,H_0)$,}
\\
\noalign{\medskip}
DV(Du) \cdot \nu = 0 &\mbox{on $B_R(O,H_0) \cap \partial \Sigma \setminus \{O\}$.}
\end{cases}
\end{equation}
Furthermore, the gradient $Du_R$ is given by $Du_R(x) = -\frac1N \, H_0(x) \, \DH_0(x)$ for $x \ne 0$ and satisfies $H(Du_R(x)) = \frac1N \, H_0(x)$ for all $x \in \mathbb R^N$.
\end{proposition}

\begin{proof}
By differentiation we find $Du_R(x) = -\frac1N \, H_0(x) \, \DH_0(x)$ for $x \ne 0$, and therefore $H(Du_R(x)) \allowbreak = \frac1N \, H_0(x)$ by~\eqref{unit}. The last equality continues to hold at the origin by Proposition~\ref{smoothness}. Let us check that $u_R$ satisfies~\eqref{radial_Dirichlet} in the weak sense. Of course, $u_R$ vanishes by definition when $H_0(x) = R$. Since $u_R \in C^1(\mathbb R^N)$ and the boundary of $\Omega =  B_R(0,H_0)$ also belongs to the class~$C^1$, it follows that $u_R \in W^{1,2}_{\Gamma_0}(\Omega \cap \Sigma)$, where $\Gamma_0$ is as in~\eqref{Gamma01_def}. Furthermore, by~\eqref{DV}, \eqref{differential} and~\eqref{unit} we have $DV(Du_R(x)) = H(Du_R(x)) \, \DH(Du_R(x)) = -\frac1N \, H_0(x) \, \DH(\DH_0(x))$. But then by Lemma~\ref{radial} it follows that $DV(Du_R(x)) = -\frac1N \, x$. We note in passing that $DV(Du_R(x)) \cdot \nu = 0$ pointwise on~$\Omega \cap \partial \Sigma \setminus \{O\}$. To complete the proof we have to show that \eqref{weak} holds. This is peculiar because, although $u_R$ may fail to have second derivatives, the compound function $DV(Du_R(x)) = -\frac1N \, x$ belongs to $C^\infty(\mathbb R^N,\mathbb R^N)$, and therefore by the divergence theorem we have
$$
\int_{\Omega \cap \Sigma}
Dv(x) \cdot DV(Du(x)) \, dx
=
-\frac1{\, N \,} \int_{\Omega \cap \Sigma}
Dv(x) \cdot x \, dx
=
\int_{\Omega \cap \Sigma}
v(x) \, dx
$$%
for every $v \in W^{1,2}_{\Gamma_0}(\Omega \cap \Sigma)$, as claimed.
\end{proof}

\begin{remark}
In the case when $\Sigma = \mathbb R^N$, the solution in the Wulff shape is considered, for instance, in \cite[(1.8)]{Bianchini-Ciraolo} and \cite[Theorem~2.1]{Ferone-Kawohl}.
\end{remark}

\section{Proof of Theorem~\ref{thm_main1}}\label{proof}

Roughly speaking, Theorem~\ref{thm_main1} asserts that if $q(r)$ grows
faster than~$r$ then the solvability of Problem~\eqref{special_pb}-\eqref{overd_cond} implies that $\Omega$ is a Wulff shape centered at the origin. The minimal rate of increase of~$q(r)$ in order to get the result is discovered by letting $R$ vary in Problem~(\ref{radial_Dirichlet}): more precisely, using Proposition~\ref{solution_WS} we find that $H(Du_R(x)) = R/N$ for every $x \in \partial B_R(O,H_0)$, hence the value of $H(Du_R(x))$ at $x \in \partial B_R(O,H_0)$ is proportional to~$R$. This information is transferred to Problem~\eqref{special_pb} by means of the following comparison argument.

\begin{proof}[Proof of Theorem~\ref{thm_main1}]
\textit{Preliminaries.} Define
$$
R_1
=
\min_{z \in \overline \Gamma_0} H_0(z)
,
\qquad
R_2
=
\max_{z \in \overline \Gamma_0} H_0(z)
$$%
and let $u_i$, $i = 1,2$, be the solution of the Dirichlet
problem~(\ref{radial_Dirichlet}) in the Wulff shape $\Omega_i =
B_{R_i}(0,H_0)$. Thus, $\Sigma \cap \Omega_1 \subseteq \Sigma \cap \Omega \subseteq \Sigma \cap \Omega_2$. We
aim to prove that $\Omega_1 = \Omega_2$, which implies the claim of
the theorem. To this purpose, pick $z_i \in \overline \Gamma_0 \cap
\partial \Omega_i$ and observe that $R_i = H_0(z_i)$, $i = 1,2$. Using Lemma \ref{nonnegativity} and 
Lemma~\ref{monotonicity} twice, we get
\begin{equation}\label{ordered}
\mbox{$u_1 \le u$ a.e.\ in $\Sigma \cap \Omega_1$,}
\qquad
\mbox{$u \le u_2$ a.e.\ in $\Sigma \cap \Omega$.}
\end{equation}
\textit{Part~1.} Taking into account that $u_1(z_1) = u(z_1) = 0$ and $u_1$ is continuously differentiable up to~$z_1$, let us check that the first inequality in~\eqref{ordered} implies
\begin{equation}\label{claim1}
\frac{\, R_1 \,}N = H(Du_1(z_1)) \leq q(R_1) \,.
\end{equation}
Letting $x(t) = z_1 - t \, |z_1|^{-1} \, z_1 \in \overline \Sigma \cap \Omega_1$ for $t \in (0, |z_1|)$, we compute the limit
$$
\ell = \lim_{t \to 0^+}
\frac{\, u_1(x(t)) \,}t
$$%
following two different arguments. On the one side, the limit $\ell$ is the radial derivative $\ell = -|z_1|^{-1} \, z_1 \cdot Du_1(z_1)$, and using Proposition~\ref{solution_WS} and equality~\eqref{radial_derivative} we may write
$$
\ell = \frac1N \, |z_1|^{-1} \, R_1^2
.
$$
On the other side, by the mean-value theorem we have $u(x(t)) = -t \, |z_1|^{-1} \, z_1 \cdot Du(\tilde x)$ for a convenient point $\tilde x$ on the segment from $z_1$ to~$x(t)$. Letting $\xi = Du(\tilde x)$ and $x = t \, |z_1|^{-1} \, z_1$ in~\eqref{dual}, and since $H_0(z_1) = R_1$, we may estimate $u(x(t)) \le t \,R_1 \, |z_1|^{-1} \, H(Du(\tilde x))$. Recalling that $u_1 \le u$ by~\eqref{ordered}, and using assumption~\eqref{overd_cond} we arrive at $\ell \le R_1 \, |z_1|^{-1} \, q(R_1)$ and \eqref{claim1} follows.

\textit{Part~2.} By using the second inequality in~(\ref{ordered}), and since $u_2$ is continuously differentiable, taking assumption~\eqref{overd_cond} into account we now prove the inequality
\begin{equation}\label{claim2}
q(R_2) \leq H(Du_2(z_2)) =  \frac{\, R_2 \,}N
.
\end{equation}
The argument is by contradiction: suppose there exists $\epsilon_0 \in (0, |z_2|)$ such that $H(Du(x)) > H(Du_2(z_2)) + \epsilon_0$ for all $x \in U_0 = \{\, x \in \Sigma \cap \Omega : |x - z_2| < \epsilon_0 \,\}$, and choose $x_0 \in U_0$. Observe that $w_0 = u_2(x_0) - u(x_0) > 0$ because the equality $w_0 = 0$ together with $u \le u_2$ implies $Du_2(x_0) = Du(x_0)$, which is not the case. While $\epsilon_0$ is kept fixed, the point~$x_0$ will tend to~$z_2$ in the end of the argument. Since the vector field $\DH(\xi)$ is continuous by assumption in $\mathbb R^N \setminus \{O\}$, for every choice of $x_0 \in U_0$ there exists a local solution $x(t)$, $t > 0$, of the initial-value problem
$$
\begin{cases}
x'(t) = \DH(Du(x(t))),
\\
\noalign{\medskip}
x(0) = x_0.
\end{cases}
$$%
Since the Euclidean norm $|x'(t)|$ is bounded from above by some constant~$M_0$, the length of the arc $\gamma$ described by $x(\tau)$ when $\tau$ ranges in the interval $(0,t)$ satisfies $|\gamma| \le M_0 \, t$, and therefore
\begin{equation}\label{bound1}
|x(t) - z_2| \le M_0 \, t
.
\end{equation}
We claim that the curve $\gamma$ can be extended until $|x(t_0) - z_2| = \epsilon_0$ for some finite $t_0 > 0$. Indeed, by differentiation of $u(x(t))$ we find $du/dt = Du(x(t)) \cdot x'(t)$. Letting $\xi = Du(x(t))$ in Lemma~\ref{radial} we see that the vector $x'(t)$ realizes the supremum in~\eqref{dual}, i.e., we may write the equality
$$
H(Du(x(t)))
=
\frac{\, Du(x(t)) \cdot x'(t) \,}{H_0(x'(t))}
.
$$%
This and~\eqref{unit} imply $du/dt = H(Du(x(t)))$. Since $H(Du(x)) > \epsilon_0$ in~$U_0$, it follows that $u$ increases along~$\gamma$ and therefore the curve, which starts at $x_0 \in U_0$, cannot end on~$\overline \Gamma_0$ where $u = 0$. Similarly, we find $du_2/dt = Du_2(x(t)) \cdot x'(t) \le H(Du_2(x(t)))$ and therefore
\begin{equation}\label{rate}
\frac{\, du_2 \,}{dt} + \epsilon_0 < \frac{\, du \,}{dt}
,
\end{equation}
hence the difference $w(t) = u_2(x(t)) - u(x(t))$ satisfies $dw/dt < -\epsilon_0$. Since $w(t)$ must be positive as long as $x(t) \in U_0$, it follows that
\begin{equation}\label{bound2}
t < w_0/\epsilon_0
\end{equation}
and the length of~$\gamma$ is bounded from above by
\begin{equation}\label{length}
|\gamma| \le M_0 \, w_0 / \epsilon_0
.
\end{equation}
In the case when $\gamma$ reaches prematurely~$\Gamma_1 \subseteq \partial \Sigma$, i.e., if $x(t) \in U_0$ for $t \in [0,t_1)$ and $x(t_1) = x_1 \in \Gamma_1$, with $|x_1 - z_2| < \epsilon_0$, the assumption that $\partial \Sigma$ is of class~$C^1$ ensures the existence of a local solution $x(t) \in \partial \Sigma$, $t > t_1$, to the following initial-value problem:
\begin{equation}\label{surface_problem}
\begin{cases}
x'(t) = \DH(Du(x(t))),
\\
\noalign{\medskip}
x(t_1) = x_1.
\end{cases}
\end{equation}
Indeed, the third condition in~\eqref{special_pb} implies that $\DH(Du(x))$ is a tangent vector to~$\partial \Sigma$ as long as $x \in \partial \Sigma$, and therefore problem~\eqref{surface_problem} admits a local solution lying on the hypersurface~$\partial \Sigma$ and extending the curve~$\gamma$. Since the curve $\gamma$, possibly extended as above, has a finite length by~\eqref{length}, and cannot end either on~$\overline \Gamma_0$ nor on $\Gamma_1$ as long as $|x(t) - z_2| < \epsilon_0$, nor can it have a limiting point in~$U_0$ because $x'(t)$ keeps far from zero and the parameter~$t$ is bounded by~\eqref{bound2}, there must be some finite $t_0 > 0$ such that $|x(t_0) - z_2| = \epsilon_0$, and therefore $|\gamma|\ge \epsilon_0$. This and~\eqref{length} yield the estimate $\epsilon_0^2 \le M_0 \, w_0$, which is uniform in the sense that $M_0$ and~$\epsilon_0$ do not depend on the choice of $x_0 \in U_0$. To conclude the argument, we now let $x_0 \to z_2$: thus, $w_0 \to 0$ while $\epsilon_0$ and~$M_0$ do not change, and a contradiction is reached.

\textit{Conclusion.} By \eqref{claim1} and~\eqref{claim2} we deduce
$$
\frac{\, q(R_2) \,}{\, R_2 \,}
\le
\frac{\, q(R_1) \,}{\, R_1 \,}
.
$$
Since the ratio $q(r)/r$ strictly increases, we must have $R_1 = R_2$ and $\Omega_1 = \Omega_2= \Omega$, as claimed.
\end{proof}

\appendix

\section{A smooth norm with a non-smooth dual norm}\label{example}

A simple example of norm in~$\mathbb R^N$ is the $p$-norm $| x
|_p$ given by
$$
| x |_p
=
\Big(\sum\limits_{k = 1}^N |x_k|^p\Big)^\frac1p\
\mbox{for $p \in [1,+\infty)$}
.
$$
In the special case when $p \in (1,+\infty)$, the dual norm of $|
x |_p$ is the $q$-norm $| \xi |_q$, where $q$ is related
to~$p$ by the equality $\frac1p + \frac1q = 1$. Both norms belong to
the class $C^1(\mathbb R^N \setminus \{O\})$. The dual norm of
$| x |_1$, instead, is $| \xi |_\infty = \max\{\,
|\xi_1|, \ldots, |\xi_N| \,\}$. Neither of the last two norms belongs
to the class $C^1(\mathbb R^N \setminus \{O\})$. In this section
we construct an explicit example of a norm $H_0 \in C^1(\mathbb R^2
\setminus \{O\})$ whose dual norm $H$ does not belong to the
same class. The example below should be compared with \cite[Example A. 1.19]{Cannarsa-Sinestrari}.

\begin{definition}\label{norm}
The norm $H_0$ is defined as the \textit{gauge function},
also called \textit{the Minkowski functional} (see
\cite[Definition~6.3.11]{Krantz-Parks} or \cite[Remark~1,
p.~380]{Simon}), of a convenient, convex, plane domain which plays the
role of the unit ball $B_1(0,H_0)$. Such a ball is defined as the
convex envelope of the four Euclidean discs of radius~$\frac12$
centered at $(\pm \frac12, 0)$, $(0, \, \pm \frac12)$ (see
Figure~\ref{flower}). Notice that the origin belongs to the boundary of
each of the given discs.
\end{definition}

Since the boundary of the ball $B_1(0,H_0)$
defined above is a $C^1$-curve, by Lemma~\ref{differentiability}
the norm $H_0$ belongs to the class $C^1(\mathbb R^2 \setminus \{O\})$. However, we have:

\begin{lemma}\label{appendix}
Denote by $H(\xi)$ the dual norm of the norm $H_0(x)$ in Definition~\ref{norm}.
The unit ball $B_1(0,H)$ is the intersection of four convex open sets,
each bounded by a parabola with focus at the origin and vertex at one
of the four points $(\pm 1,0)$, $(0, \, \pm 1)$.
\end{lemma}

Before proving the lemma, we notice that the boundary $\partial B_1(0,H)$
of the unit ball described in the statement has a corner at each of
the four points $(\pm \overline \xi, \, \pm \overline \xi)$,
$\overline \xi = 2 \, (\sqrt{2 \,} - 1)$. To see this, let us consider
the parabola $\gamma$ with focus at~$O$ and vertex at $(1,0)$. This
line is the graph of the function $\xi_1(\xi_2) = 1 - \frac14 \,
\xi_2^2$ whose derivative satisfies $\xi_1'(\overline \xi) = -\frac12
\, \overline \xi = 1 - \sqrt{2 \,} > -1$. Hence $\gamma$ is not
orthogonal to the straight line $\xi_1 = \xi_2$ at the point of
intersection $(\overline \xi, \overline \xi)$, and consequently the
boundary $\partial B_1(0,H)$ must have a corner there. Thus, the dual
norm~$H$ of the given norm~$H_0$ does not belong to the class
$C^1(\mathbb R^2 \setminus \{O\})$.

\begin{figure}[h]
\centering
\begin{picture}(150,130)(73,4)
\includegraphics{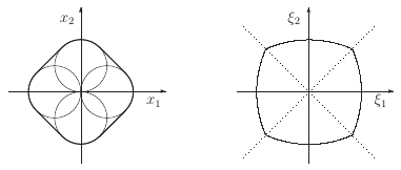}
\end{picture}
\caption{The ball $B_1(0,H_0)$ (left) is smooth, its dual (right) is not.}
\label{flower}
\end{figure}

\begin{proof}[Proof of Lemma~\ref{appendix}]
Let us describe the boundary of the ball $B_1(0,H)$ in parametric
form. Passing to polar coordinates $\rho,\vartheta$ related to
$\xi_1,\xi_2$ by $\xi_1 = \rho \, \cos \vartheta$, $\xi_2 = \rho \,
\sin \vartheta$, for every $\vartheta \in (-\frac\pi4, \frac\pi4)$ we
compute the Euclidean norm $\rho(\vartheta) = | v_\vartheta |$
of the unique vector $v_\vartheta = \rho(\vartheta) \, (\cos
\vartheta, \allowbreak \sin \vartheta)$ satisfying $H(v_\vartheta) =
1$. To this purpose it is enough to locate the point $P_\vartheta \in
\partial B_1(0,H_0)$ where the outer normal~$\nu$ equals $(\cos
\vartheta, \sin \vartheta)$: indeed, due to~\eqref{unit}, we have $H(\DH_0(P_\vartheta)) = 1$ and therefore we
may take $v_\vartheta = \DH_0(P_\vartheta)$.
Recalling that $H_0(P_\vartheta) = 1$, the radial derivative $\partial
H_0/\partial r$, $r = \sqrt{x_1^2 + x_2^2 \,}$, is easily computed at
$P_\vartheta$ by~(\ref{radial_derivative}) and~(\ref{differential}):
$$
\frac{\, \partial H_0 \,}{\, \partial r \,}(P_\vartheta)
=
| P_\vartheta |^{-1}
.
$$%
Now the construction of $B_1(0,H_0)$ comes into play: since the
origin~$O$ and the point~$P_\vartheta$ belong to the circumference of
radius~$\frac12$ centered at $(\frac12, 0)$, by a classical theorem in
Euclidean geometry we get that the line segment $OP_\vartheta$ makes
an angle $\alpha = \vartheta/2$ with the $x_1$-axis (see
Figure~\ref{Euclid}), and therefore $| P_\vartheta | = \cos
\alpha = \cos(2\vartheta)$.
\begin{figure}[h]
\centering
\begin{picture}(200,160)(12,5)
\includegraphics{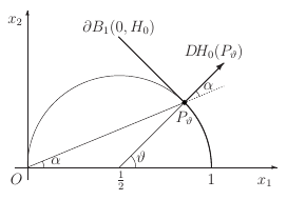}
\end{picture}
\caption{Finding the Euclidean norm of $\DH_0(P_\vartheta)$}
\label{Euclid}
\end{figure}
Finally, since the radial derivative is the projection of the gradient
in the radial direction, we have
$$
\frac{\, \partial H_0 \,}{\, \partial r \,}(P_\vartheta)
=
| \DH_0(P_\vartheta) | \, \cos \alpha,
$$%
and hence
$$
| \DH_0(P_\vartheta) |
=
\frac1{\, \cos \alpha \,}
\,
\frac{\, \partial H_0 \,}{\, \partial r \,}(P_\vartheta)
=
\frac2{\, 1 + \cos \vartheta \,}
.
$$%
Thus, the components $\xi_1(\vartheta),\xi_2(\vartheta)$
of~$v_\vartheta$ are given by
$$
\begin{cases}
\xi_1(\vartheta)
=
\displaystyle
\frac{2 \, \cos \vartheta}{\, 1 + \cos \vartheta \,};
\\
\noalign{\medskip}
\xi_2(\vartheta)
=
\displaystyle
\frac{2 \, \sin \vartheta}{\, 1 + \cos \vartheta \,}.
\end{cases}
$$%
The parametric equations given above describe the parabola whose
Cartesian equation is $\xi_1 = 1 - \frac14 \, \xi_2^2$, which passes
through the points $(0, \, \pm 2 \, (\sqrt{2 \,} - 1))$ and has focus
at the origin and vertex at~$(1,0)$. The remaining parts of $\partial
B_1(0,H)$ are managed similarly, and the lemma follows.
\end{proof}

\pdfbookmark{Acknowledgments}{Acknowledgments}
\section*{Acknowledgments} The authors are members of the Gruppo
Nazionale per l'A\-na\-li\-si Matematica, la Pro\-ba\-bi\-li\-t\`a e
le loro Applicazioni
(\href{https://www.altamatematica.it/gnampa/}{GNAMPA}) of the Istituto
Na\-zio\-na\-le di Alta Ma\-te\-ma\-ti\-ca
(\href{https://www.altamatematica.it/en/}{INdAM}). G. Ciraolo has been partially supported by the PRIN 2017 project ``Qualitative and quantitative aspects of nonlinear PDEs''.
A.\ Greco is
partially supported by the research project {\em Evolutive and stationary Partial Differential Equations with a focus on bio-mathematics}, funded by
\href{https://www.fondazionedisardegna.it/}{Fondazione di Sardegna}
(2019). 

\pdfbookmark{References}{References}

\bigskip

\hbox{
\vbox to 0pt{\noindent Giulio Ciraolo\\
  Department of Mathematics ``Federigo Enriques''\\
  Universit\`a degli Studi di Milano\\
  Italy\\
  e-mail: giulio.ciraolo@unimi.it\vss}}%
\hbox to 1.55\textwidth{\hss
\vbox to 21.4mm{\vss\noindent Antonio Greco\\
  Department of Mathematics and Computer Science\\
  Universit\`a degli Studi di Cagliari\\
  Italy\\
  e-mail: greco@unica.it}}

\end{document}